\documentclass[a4paper,12pt,reqno]{amsproc}
\usepackage{amssymb}
\usepackage{multirow}
\usepackage[english]{babel} 
\usepackage[latin1]{inputenc}
\usepackage{fullpage}
\newtheorem*{proposition}{Proposition}

\newcommand{\imat}{\mathbf{I}}
\newcommand{\abs}[1]{\lvert#1\rvert}
\newcommand{\Hom}{\mathrm{hom}}

\DeclareMathOperator{\PGL}{PGL}
\DeclareMathOperator{\GL}{GL}
\DeclareMathOperator{\PHG}{PHG}
\DeclareMathOperator{\PG}{PG}
\DeclareMathOperator{\GR}{GR}
\DeclareMathOperator{\rad}{rad}

\newcommand{\mset}[1]{\mathfrak{#1}}
\newcommand{\nmax}{\mathrm{m}}
 
\newcommand{\Z}{\mathbb{Z}}

\newcommand{\F}{\mathbb{F}}
\newcommand{\G}{\mathbb{G}}
\renewcommand{\H}{\mathbb{H}}
\newcommand{\I}{\mathbb{I}}
\newcommand{\J}{\mathbb{J}}
\newcommand{\K}{\mathbb{K}}
\renewcommand{\S}{\mathbb{S}}
\newcommand{\T}{\mathbb{T}}

\newcommand{\wrt}{w.\,r.\,t.}  
\newcommand{\ie}{i.\,e.}

\title[Classification of Maximal Arcs]{Classification of Maximal Arcs in
 Small Projective Hjelmslev Geometries}

\author{Thomas Honold}
\address{Thomas Honold\\
  Technische Universität München\\
  Zentrum Mathematik (M11)\\
  Boltz\-mann\-str.~3\\
  D-85748 Garching\\
  Germany} \email{honold@ma.tum.de}

\author{Michael Kiermaier}
\address{Michael Kiermaier\\
    Technische Universität München\\
  Zentrum Mathematik (M11)\\
  Boltz\-mann\-str.~3\\
  D-85748 Garching\\
  Germany} \email{michael.kiermaier@gmx.net}

\date{}

\begin{document}

\begin{abstract}
  Maximal arcs in small projective Hjelmslev geometries
  are classified up to isomorphism, and the parameters of the
  associated codes are determined.
\end{abstract}

\keywords{Codes over chain rings, projective Hjelmslev geometry,
  maximal arc, nauty}

\maketitle

\section{Introduction}\label{sec:intro}

Fat\footnote{A linear code is \emph{fat} if its coordinate
  projections are onto.}
linear codes over a finite chain ring $R$ may be viewed
as multisets of points in projective Hjelmslev geometries over $R$ as
developed in~\cite{it:chain}. There is evidence that, just as in the classical
case $R=\F_q$, good (from a coding theorist's point-of-view) linear
codes over $R$ correspond in general to interesting (from a geometer's
point-of-view) multisets of points in projective Hjelmslev geometries.

This paper reports on a computer search for maximal arcs in projective
Hjelmslev geometries of small dimension over chain rings of order at
most $16$. In the
cases under consideration the arcs have been classified up to
geometric equivalence. 

We refer to \cite{ivan-silvia04,it:kent,it:ovals,it:deadfin} for the
combinatorics of projective Hjelmslev geometries and in particular for
results on arcs in projective Hjelmslev planes. Some general
properties of the MDS-like codes associated with arcs can be found in
\cite{t:mdslike}.

With a finite chain ring $R$ (\emph{not} assumed to be commutative) we
associate the pair of parameters $(q,m)$, where $R/\rad
R\cong\F_q=\F_{p^r}$ and $m$ is the composition length of ${}_RR$ (or $R_R$).
Then $R$ has cardinality $q^m$ and characteristic $p^\lambda$, where
$1\leq\lambda\leq m$. Up to isomorphism there is only one ring with
$\lambda=m$, the Galois ring $\G_{q,m}=\GR(q^m,p^m)$, and one
commutative ring with $\lambda=1$, the truncated polynomial ring
$\S_{q,m}=\F_q[X]/(X^m)$. The remaining chain rings of order $\leq 16$
are $\H_8=\Z_4[X]/(X^2+2,X^3)$, the truncated skew polynomial ring
$\T_4=\F_4[X;\sigma]/(X^2)$ \wrt\ $\sigma\colon\F_4\to\F_4$, $a\mapsto
a^2$, and the rings $\I_{16}=\Z_4[X]/(X^2+2)$,
$\J_{16}=\Z_4[X]/(X^2+2X+2)$, $\K_{16}=\Z_4[X]/(X^3+2,X^4)$. These
rings have orders as indicated by the subscript.\footnote{No rule
  without exception. The second author denies any responsibility for
  the weird `$\T_4$'.} 

Denoting by $R$ one of the rings listed above, we write $\PHG(k,R)$
for the $k$-dimensional projective Hjelmslev geometry over
$R$.\footnote{Even in the case of $\T_4$ there is no need to
  distinguish between ``left'' and ``right'' geometries. The
  antiautomorphism $a_0+a_1X\mapsto a_0+\sigma(a_1)X$ of $\T_4$
  identifies $\PHG\bigl({}_{\T_4}(\T_4^l)\bigr)$ with
  $\PHG\bigl((\T_4^l)_{\T_4}\bigr)$.} The points (lines, hyperplanes)
of $\PHG(k,R)$ are the free rank $1$ (rank $2$, resp.\ rank $k$)
submodules of a free $R$-module of rank $k+1$, say $R^{k+1}_R$, and
incidence is defined by set inclusion.

Following the notation in \cite[Ch.~3.3]{hirschfeld98}, an \emph{$(n,u)$-arc}
in $\PHG(k,R)$ is an $n$-multiset of points of $\PHG(k,R)$ with at most
$u$ points (counted with their multiplicities) on any hyperplane. An
\emph{$(n,u)$-arc} is called \emph{degenerate} if its points (\ie\
those of nonzero multiplicity) generate a proper submodule of $R^{k+1}_R$. We
denote by $\nmax_u(k,R)$ the maximum $n$ for which an $(n,u)$-arc
in $\PHG(k,R)$ exists. An $(n,u)$-arc is said to be \emph{complete} if it
is not contained in an $(n+1,u)$-arc and \emph{maximal} if
$n=\nmax_u(k,R)$. 

An $n$-multiset $\mset{k}$ of points of $\PHG(k,R)$ defines an
$R$-linear code $C\leq{}_RR^n$ by taking representative vectors of the
points as columns of a generator matrix for $C$. Using a generalized
Gray map $\psi\colon R\to\F_q^{q^{m-1}}$ as in
\cite{greferath-schmidt98a}, we obtain from $C$ a (not necessarily
linear) distance-invariant $q$-ary code $\psi(C)$.\footnote{If $R$ has
  prime characteristic $p$ then $\psi(C)$ is in fact a linear code.}
The weight distribution of $\psi(C)$ can be computed from geometric
information on $\mset{k}$; see \cite[Th.~5.2]{it:chain}.

\section{The Arcs}\label{sec:arcs}

In \cite{royle98} G.F.~Royle describes an algorithm for the
classification of complete arcs in projective planes over finite
fields. The algorithm essentially applies B.D.~McKay's isomorph-free exhaustive
generation method for combinatorial structures (cf.\ \cite{mckay98}
and the software package \textsf{nauty} available at
\verb+http://cs.anu.edu.au/~bdm/nauty/+) to the point-line incidence
graph of the plane.

The second author~\cite{kiermaier06} has adapted Royle's algorithm to
the case $\PHG(k,R)$. In particular, the classification of complete arcs in
$\PHG(k,R)$ is accomplished by applying the method to
the point-hyperplane incidence graph of $\PHG(k,R)$.

Table~\ref{uebersichtTabelle} shows the main results of this computer
search. Tabulated are the numbers $\nmax_u(k,R)$ for the chain rings
of order $\leq 16$ and small $k$, $u$, along with the number of
equivalence classes of $(n,u)$-arcs, $n=\nmax_u(k,R)$. 
Unique (up to equivalence) maximal arcs are indicated
by bold type, otherwise the number of equivalence classes of
nondegenerate maximal arcs (the total number of equivalence
classes of maximal arcs) is appended as a subscript (resp.\ a
subscript in parentheses).

\begin{table}
  \setlength{\arraycolsep}{2pt}
  \begin{equation*}
\begin{array}{c|c||cc|ccc|cc|ccc|ccccc}
\multicolumn{2}{c||}{(q,m)} & \multicolumn{2}{|c|}{(2,2)} & \multicolumn{3}{|c|}{(2,3)} & \multicolumn{2}{|c|}{(3,2)} & \multicolumn{3}{|c|}{(4,2)} & \multicolumn{5}{|c}{(2,4)} \\
\hline
k & u & \mathbb{Z}_4 & \mathbb{S}_{2,2} & \mathbb{Z}_8 & \mathbb{H}_8 & \mathbb{S}_{2,3} & \mathbb{Z}_9 & \mathbb{S}_{3,2} & \mathbb{G}_{4,2} & \mathbb{S}_{4,2} & \mathbb{T}_{4} & \mathbb{Z}_{16} & \mathbb{I}_{16} & \mathbb{J}_{16} & \mathbb{K}_{16} & \mathbb{S}_{2,4}\\
\hline
\hline
\multirow{5}{*}{2} & 2 & \mathbf{7} & 6_2 & \mathbf{10} & 10_5 & 10_5 & 9_3 & 9_4 & \mathbf{21} & 18_6 & \mathbf{18} & \geq 16 & \mathbf{22} & \mathbf{22} & 19_5 & \mathbf{19}\\
& 3 & 10_8 & 10_8 & \mathbf{21} & 18_{93} & 18_{93} & 19_3 & 18_{255} & & & & &\\
& 4 & 16_3 & 16_3 & & & & & & & & & & & &\\
& 5 & \mathbf{22} & \mathbf{22} & & & & & & & & & & & &\\
& 6 & \mathbf{28} & \mathbf{28} & & & & & & & & & & & &\\
\hline
\multirow{3}{*}{3} & 3 & \mathbf{8} & 6_{1(2)} & 8_{57(68)} & \mathbf{9} & \mathbf{9} & \mathbf{10} & \mathbf{10} & & & & & & & &\\
& 4 & 10_{25} & \mathbf{11} & & & & & & & & & & & & &\\
& 5 & 16_2 & 16_2 & & & & & & & & & & & & &\\
\hline
\multirow{2}{*}{4} & 4 & 6_{5(17)} & 6_{5(17)} & & & & & & & & & & & & &\\
& 5 & 11_4 & 11_6 & & & & & & & & & & & & &\\
\end{array}
\end{equation*}
\caption{Cardinalities and number of equivalence classes of maximal
  $(n,u)$-arcs in $\PHG(k,R)$, $\abs{R}\leq 16$}
\label{uebersichtTabelle}
\end{table}

Selected maximal arcs from Table~\ref{uebersichtTabelle} are listed
in the appendix, along with data on the Gray images of the associated
$R$-linear codes.


\section{Remarks}\label{sec:rmk}

  For chain rings of length $m=2$, the numbers $\nmax_u(2,R)$ were
  known previously except for the cases
  $\nmax_2(2,\S_{2,2})=\nmax_2(2,\T_4)=18$; cf.\
  \cite{ivan-silvia04,it:kent,it:ovals,it:deadfin}.

  It has been conjectured before that for given $q$, $m$, $k$, $u$,
  the numbers $\nmax_u(k,R)$ form a non-decreasing function of the
  characteristic of $R$. This is not true in general, as the
  example of the (unique) $(11,4)$-arc in $\PHG(3,\S_{2,2})$ shows which is
  bigger than any $(n,4)$-arc in $\PHG(3,\Z_4)$.

  A $(q^m+q^{m-1}+1,2)$-arc in $\PHG(2,R)$ is also referred to as
  a \emph{hyperoval}. Through any point of $\PHG(2,R)$ there are
  $q^m+q^{m-1}$ lines, so hyperovals have no tangents.  It is known
  that hyperovals in $\PHG(2,R)$ exist iff $R=\F_{2^r}$ (a finite
  field of even order) or $R=\G_{2^r,2}$ (a Galois ring of
  characteristic $4$): The case $m=1$ is classical, $m=2$ has been
  done in \cite{it:ovals}, and for $m\geq 3$ the nonexistence of
  hyperovals in $\PHG(2,R)$ follows from the observation that any two
  points of a hyperoval are on a unique secant (so the
  number of points cannot exceed $q^2+q+1$).

  According to Table~\ref{uebersichtTabelle} hyperovals in the planes
  over $\Z_4$ and $\G_{4,2}$ are unique up to equivalence. The
  following proposition gives a bit more information on hyperovals in
  $\PHG(\Z_4)$.
  \begin{proposition}
    \label{prop:7}
    The set $\mathfrak{H}$ of hyperovals of $\PHG(2,\Z_4)$ has
    cardinality $256$. The
    automorphism group $G$ of $\PHG(2,\Z_4)$ is transitive on
    $\mathfrak{H}$ and the stabilizer $G_{\mset{h}}$ of a hyperoval
    $\mset{h}\in\mathfrak{H}$ has order $168$. Further, $G$ has a
    normal subgroup $H$ which is regular on $\mathfrak{H}$.
  \end{proposition}
  \begin{proof}
    The group $\PGL(3,\Z_4)$ acts regularly on ordered quadrangles
    (four points in different neighbour classes, no three neighbour classes
    on a line of the quotient plane $\PG(2,\F_2)$). The 
    stabilizer in $G$ of an ordered quadrangle is easily seen to be
    trivial. Hence $G\cong\PGL(3,\Z_4)$ (an instance of the
    Fundamental Theorem of Projective Hjelmslev Geometry) and
    $\abs{G}=\abs{\PGL(3,\Z_4)}=2^9\cdot\abs{\GL(3,2)}/2=256\cdot 168$. Now
    observe that a quadrangle is contained in a unique hyperoval---for
    the canonical quadrangle $\Z_4(100)$, $\Z_4(010)$, $\Z_4(001)$,
    $\Z_4(111)$ the remaining points are $\Z_4(123)$, $\Z_4(312)$ and
    $\Z_4(231)$---and a hyperoval contains $168$ ordered
    quadrangles---as many as the quotient plane $\PG(2,\F_2)$. This
    yields all except the last assertion.  Finally,
    since each $G_{\mset{h}}\cong\GL(3,2)$ acts faithfully on the quotient
    plane, the required normal subgroup
    of $G$ is the kernel of the action of $G$ on the quotient
    plane, \ie\ $H=(\imat_3+2\Z_4^{3\times 3})/\{\pm\imat_3\}$, where
    $\imat_3$ denotes the $3\times 3$ identity matrix.
  \end{proof}

  We leave it as an exercise to prove the uniqueness of the $(8,3)$-arc of
  $\PHG(3,\Z_4)$ (which corresponds to the $\Z_4$-linear
  Nordstrom-Robinson code) along similar lines.

\appendix

\section*{Appendix}\label{sec:app}

Data on selected arcs from Table~\ref{uebersichtTabelle} is given
below. Each entry contains the arc (in homogeneous coordinates), the
order $g$ of its automorphism group, the minimum homogeneous distance
$d_\Hom$ of the associated code $C$ (which is equal to $q^{2-m}$ times
the minimum Hamming distance of the Gray image) and information on the
$q$-ary Gray image $\psi(C)$.

\begin{description}
\addtolength{\itemsep}{1ex}
\item[$(7,2)$-arc in $\PHG(2,\Z_4)$]\ \\  
  $(1:0:0),(0:1:0),(0:0:1),(1:1:1),(2:1:3),(1:2:3),(1:3:2)$\\
  $g=168$, $d_\Hom=6$\\
  $\psi$ gives a binary $[14,6,6]$-code with weight enumerator $1 + 42
  X^6 + 7 X^8 + 14 X^{10}$.  The best linear binary $[14,6]$-code has
  minimum distance $5$.
\item[$(22,5)$-arc in $\PHG(2,\Z_4)$]\ \\
  $g=1536$, $d_\Hom=20$\\
  $(1:0:0),(0:1:0),(0:0:1),(1:1:1),(0:1:1),(1:0:1),(1:1:0),\\
  (1:2:2),(2:1:2),(2:2:1),(1:3:3),(2:1:3),(1:2:3),(1:3:2),\\
  (1:2:0),(2:1:0),(2:0:1),(1:3:1),(1:0:2),(0:1:2),(0:2:1),\\
  (1:1:3)$\\
  $\psi$ gives a binary $[44,6,20]$-code with weight enumerator $1 + 6
  X^{20} + 48 X^{22} + 6 X^{24} + 2 X^{28} + X^{32}$.  There is a
  linear binary $[44,6,21]$-code.
\item[$(22,5)$-arc in $\PHG(2,\S_{2,2})$]\ \\
  $g=1536$, $d_\Hom=20$\\
  $(1:0:0),(0:1:0),(0:0:1),(1:1:1),(0:1:1),(1:0:1),(1:X:X),\\
  (X:1:X),(X:X:1),(1:X+1:X+1),(X:1:X+1),(1:X:X+1),\\
  (1:X:0),(X:1:0),(X:0:1),(1:X+1:1),(1:X+1:0),\\
  (1:0:X),(0:1:X),(0:X:1),(1:1:X+1),(1:1:X)$\\
  $\psi$ gives a linear binary code with the same parameters as in the last
  case.
\item[$(8,3)$-arc in $\PHG(3,\Z_4)$]\ \\ 
  $g=1344$, $d_\Hom=6$\\
  $(1:0:0:0),(0:1:0:0),(0:0:1:0),(0:0:0:1),(1:3:3:2),\\
  (2:1:3:3),(1:2:1:3),(1:1:2:1)$\\
  $\psi$ gives a binary $[16,8,6]$-code with weight enumerator $1 +
  112 X^6 + 30 X^8 + 112 X^{10} + X^{16}$.  The best linear binary
  $[16,8]$-code has minimum distance $5$.
\item[$(11,4)$-arc in $\PHG(3,\S_{2,2})$]\ \\
  $g=24$, $d_\Hom=8$\\
  $(1:0:0:0),(0:1:0:0),(0:0:1:0),(0:0:0:1),(1:1:1:1),\\
  (X:1:X+1:X),(1:X:X+1:X+1),(1:X:1:X),\\
  (1:X+1:0:X),(1:X+1:X+1:0),(0:1:X:X+1)$\\
  $\psi$ gives a linear binary $[22,8,8]$-code with weight enumerator $1 + 54
  X^8 + 76 X^{10} + 72 X^{12} + 48 X^{14} + X^{16} + 4 X^{18}$ which
  is optimal.
\item[$(10,2)$-arc in $\PHG(2,\Z_8)$]\ \\
  $g=8$, $d_\Hom=6$\\
  $(1:0:0),(0:1:0),(0:0:1),(1:1:1),(2:1:3),(4:1:5),(1:5:4),\\
  (6:1:2),(1:7:2),(1:3:5)$\\
  $\psi$ gives a binary $[40,9,12]$-code with weight enumerator $1 + 4
  X^{12} + 70 X^{16} + 128 X^{18} + 168 X^{20} + 32 X^{22} + 72 X^{24}
  + 32 X^{26} + 4 X^{28} + X^{32}$.  There is a linear binary
  $[40,9,16]$-code.
\item[$(21,3)$-arc in $\PHG(2,\Z_8)$]\ \\
$g=168$, $d_\Hom=18$\\
$(1:0:0),(0:1:0),(1:2:2),(2:2:1),(2:1:3),(1:2:1),(1:1:3),\\
(1:1:2),(1:5:5),(1:5:4),(1:7:6),(6:1:5),(4:1:7),(6:1:0),\\
(1:7:1),(0:4:1),(1:0:5),(2:1:6),(1:2:7),(1:0:6),(0:6:1)$\\
$\psi$ gives a binary $[84,9,36]$-code with weight enumerator $1 + 14 X^{36} + 168 X^{38} + 196 X^{42} + 42 X^{44} + 7 X^{48} + 84 X^{50}$.
There is a linear binary $[84,9,38]$-code.
\item[$(9,3)$-arc in $\PHG(3,\H_8)$]\ \\
$g=12$, $d_\Hom=5$\\
$(1:0:0:0),(0:1:0:0),(0:0:1:0),(0:0:0:1),(1:X:X+1:X),\\
(X:X:1:X+1),(1:X+3:2:X+2),(2:1:X+3:X+3),\\
(1:2:X+3:X+1)$\\
$\psi$ gives a binary $[36,12,10]$-code with weight enumerator $1 + 12 X^{10} + 166 X^{12} + 504 X^{14} + 873 X^{16} + 908 X^{18} + 1020 X^{20} + 468 X^{22} + 110 X^{24} + 24 X^{26} + 6 X^{28} + 4 X^{30}$.
There is a linear binary $[36,12,12]$-code.
\item[$(9,3)$-arc in $\PHG(3,\S_{2,3})$]\ \\
$g=12$, $d_\Hom=5$\\
$(1:0:0:0),(0:1:0:0),(0:0:1:0),(0:0:0:1),(1:X:X+1:X),\\
(X^2+X:1:X^2+X+1:X^2+X),(1:X^2+X+1:X^2:X^2+X),\\
(1:X^2:X^2+X+1:X^2+X+1),(X^2:1:X+1:X^2+X+1)$\\
$\psi$ gives a linear binary code with the same parameters as in the last case.
\item[$(10,3)$-arc in $\PHG(3,\Z_9)$]\ \\
$g=10$, $d_\Hom=15$\\
$(1:0:0:0),(0:1:0:0),(0:0:1:0),(0:0:0:1),(1:1:1:1),\\
(1:3:2:4),(1:4:3:2),(1:2:8:3),(3:1:8:2),(1:8:4:5)$\\
$\psi$ gives a ternary $[30,8,15]$-code with weight enumerator $1 + 720 X^{15} + 1680 X^{18} + 3240 X^{21} + 900 X^{24} + 20 X^{27}$.
There no better linear ternary $[30,8]$-code.
\item[$(10,3)$-arc in $\PHG(3,\S_{3,2})$]\ \\
$g=10$, $d_\Hom=15$\\
$(1:0:0:0),(0:1:0:0),(0:0:1:0),(0:0:0:1),(1:1:1:1),\\
(1:X:X+2:X+1),(1:X+2:2:2X),(1:2X+2:2X+1:X+2),\\
(1:2X+1:X:2X+2),(2X:1:X+2:2X+2)$\\
$\psi$ gives an (optimal) linear ternary code with the same parameters
as in the last case.
\item[$(21,2)$-arc in $\PHG(2,\G_{4,2})$]\ \\
  $g=126$, $d_\Hom=60$\\
  $(1:0:0),(0:1:0),(0:0:1),(1:1:1),(1:3:3a),(2:1:3a),(1:3a:2a+2),\\
  (1:a+3:2a),(1:2:a+1),(1:2a+3:2),(1:2a+2:3),(2a+2:1:3a+3),\\
  (1:a:3a+1),(1:3a+3:a+2),(2a:1:3),(1:2a+1:3a+3),(1:a+2:a),\\
  (1:3a+1:2a+3),(1:2a:3a+2),(1:a+1:a+3),(1:3a+2:2a+1)$\\
  $\psi$ gives a quaternary $[84,6,60]$-code with weight
  enumerator $1 + 2520 X^{60} + 63 X^{64} + 1512 X^{68}$.  The best
  known linear quaternary $[84,6]$-code has minimum distance $59$.
\item[$(18,2)$-arc in $\PHG(2,\T_4)$]\ \\
  $g=96$, $d_\Hom=48$\\
  $(1:0:0),(0:1:0),(0:0:1),(1:1:1),(1:X+1:X+a),(1:X+a:X),\\
  (1:X:aX+1),(1:X+(a+1):a),(aX:1:X+1),(1:aX:aX+a),\\
  (1:aX+a:(a+1)X+a),(1:aX+(a+1):(a+1)X+1),\\
  (1:(a+1)X+a:X+1),(1:(a+1)X+1:aX+(a+1)),\\
  (1:(a+1)X+(a+1):aX),((a+1)X:1:aX+a),\\
  (1:a+1:X+(a+1)),(1:a:(a+1)X+(a+1))$\\
  $\psi$ gives a linear quaternary $[72,6,48]$-code with weight
  enumerator $1 + 12 X^{48} + 864 X^{50} + 960 X^{51} + 96 X^{52} +
  576 X^{54} + 144 X^{56} + 864 X^{58} + 576 X^{59} + 3 X^{64}$.
  There is a linear quaternary $[72,6,50]$-code.
\end{description}


\end{document}